\documentclass[11pt]{article}

\usepackage[utf8]{inputenc}
\usepackage[british]{babel}
\usepackage[a4paper, margin=2.5cm]{geometry}
\usepackage{amsmath, amsthm, amssymb, amsfonts}
\usepackage{mathtools}
\usepackage{thmtools}
\usepackage[shortlabels]{enumitem}
\usepackage[font={small, it}]{caption}
\usepackage{microtype}
\usepackage{xcolor}
\usepackage{mathabx}
\usepackage{tikz}
\usepackage{hyperref}
\usepackage{bm}

\newcommand{\cP}{\ensuremath{\mathcal P}}

\newcommand{\cS}{\ensuremath{\mathcal S}}

\newcommand{\eps}{\varepsilon}
\renewcommand{\phi}{\varphi}
\renewcommand{\rho}{\varrho}

\let\setminus=\smallsetminus
\let\emptyset=\varnothing

\declaretheorem[parent=section]{theorem}
\declaretheorem[sibling=theorem]{lemma}

\declaretheorem[sibling=theorem,style=definition]{definition}

\setlist{itemsep=0.1em, topsep=0.1em, parsep=0.1em, partopsep=0.1em}

\hypersetup{
  colorlinks,
  linkcolor={red!60!black},
  citecolor={green!50!black},
  urlcolor={blue!80!black}
}

\colorlet{RoyalRed}{red!70!black}
\definecolor{RoyalBlue}{rgb}{0.25, 0.41, 0.88}
\definecolor{RoyalAzure}{rgb}{0.0, 0.22, 0.66}

\newlength{\bibitemsep}
\newlength{\bibparskip}
\let\oldthebibliography\thebibliography
\renewcommand\thebibliography[1]{%
  \oldthebibliography{#1}%
  \setlength{\parskip}{\bibitemsep}%
  \setlength{\itemsep}{\bibparskip}%
}

\title{Routing permutations on spectral expanders via matchings}
\author{
Rajko Nenadov\thanks{Google Z\"urich. Email: \texttt{rajkon@gmail.com}.}
}
\date{}

\begin{document}
\maketitle

\begin{abstract}
We consider the following matching-based routing problem. Initially, each vertex $v$ of a connected graph $G$ is occupied by a pebble which has a unique destination $\pi(v)$. In each round the pebbles across the edges of a selected matching in $G$ are swapped, and the goal is to route each pebble to its destination vertex in as few rounds as possible. We show that if $G$ is a sufficiently strong $d$-regular spectral expander then any permutation $\pi$ can be achieved in $O(\log n)$ rounds. This is optimal for constant $d$ and resolves a problem of Alon, Chung, and Graham [SIAM J. Discrete Math., 7 (1994), pp. 516--530].
\end{abstract}

\section{Introduction} \label{sec:introduction}

The following routing problem was introduced by Alon, Chung, and Graham \cite{alon94routing}. Given a graph $G$, initially each vertex $v$ is occupied by a pebble $p_v$ (both vertices and pebbles are labelled). For each pebble $p_v$ we are also given its unique destination $\pi(v)$, that is, $\pi$ is a permutation on $V(G)$. One round of routing consists of selecting a matching in $G$ and swapping pebbles along each edge of such a matching. For example, if a pebble $p$ currently sits on a vertex $v$, a pebble $p'$ on a vertex $b$, and the edge $\{a, b\}$ is part of the selected matching, then we move the pebble $p$ to the vertex $b$ and the pebble $p'$ to the vertex $a$. We denote by $\textrm{rt}(G, \pi)$ the smallest number of rounds needed to achieve a permutation $\pi$, and $\mathrm{rt}(G) = \max_{\pi} \mathrm{rt}(G, \pi)$. Classes of graph for which $\mathrm{rt}(G)$ has been studied include complete (bipartite) graphs, paths, cycles, trees,  hypercubes, and expanders \cite{alon94routing,paul20routing,wei10routing,zhang99tree}. 

In this note we are interested in the case where $G$ is a \emph{spectral} expander. We say that a graph $G$ is an \emph{$(n, d, \lambda)$-graph} if it is a $d$-regular graph with $n$ vertices and $\lambda(G) \le \lambda$, where $\lambda(G)$ denotes the second largest \emph{absolute} eigenvalue of its adjacency matrix. A fundamental result in spectral graph theory, the \emph{Expander Mixing Lemma} (e.g.\ see \cite[Lemma 2.5]{hoory06expander}), demonstrates the importance of $\lambda(G)$: For any $S, T \subseteq V(G)$ we have
\begin{equation} \label{eq:mixing_lemma}
    \left| e_G(S, T) - |S||T| d / n \right| \le \lambda(G) \sqrt{|S||T|},
\end{equation}
where $e_G(S, T)$ counts the number of edges with one endpoint in $S$ and the other in $T$ (in the case $S \cap T \neq \emptyset$, every edge with both endpoints in $S \cap T$ is counted twice). The largest eigenvalue of a $d$-regular graph $G$ is always $d$, thus $\lambda(G) \le d$, and the celebrated Alon-Boppana bound (see \cite[Theorem 2.7]{hoory06expander}) states $\lambda(G) \ge 2\sqrt{d - 1} - o_n(1)$. Note that already for $\lambda(G) = o(d)$, the inequality \eqref{eq:mixing_lemma} implies strong edge and vertex-expansion properties of the graph $G$. For a thorough introduction to $(n, d, \lambda)$-graphs, expander graphs, and pseudo-random graphs in general, we refer the reader to \cite{hoory06expander,krivelevich06pseudorandom}.

It was shown in \cite{alon94routing} that if $G$ is an \emph{$(n,d,\lambda)$-graph}, for any $\lambda < d$, then
$$
    \mathrm{rt}(G) = O\left( \frac{d^2}{(d - \lambda)^2} \log^2 n \right).
$$
We improve this bound under a mild assumption on $\lambda$.

\begin{theorem} \label{thm:routing}
There exists $d_0, C \in \mathbb{N}$ such that if $G$ is an $(n, d, \lambda)$-graph with $d \in [d_0, n-1]$ and $\lambda < d/72$, then 
$$
    \textrm{rt}(G) \le \frac{C \log n}{\log(d / \lambda)}.
$$
\end{theorem}

The special case of Theorem \ref{thm:routing} where  $d \ge n^{\alpha}$ and $\lambda =  K \sqrt{d}$, for some constants $\alpha, K > 0$ (that is, $G$ is a very dense graph with, asymptotically, the strongest possible spectral expansion properties) was recently obtained by Horn and Purcilly \cite{paul20routing}. In the case where $d$ is a constant, Theorem \ref{thm:routing} gives $\mathrm{rt}(G) = O(\log n)$ which answers a problem raised by Alon, Chung, and Graham \cite{alon94routing}. In general, Theorem \ref{thm:routing} is easily seen to be optimal for $\lambda < d^{1 - \eps}$, for any constant $\eps > 0$, as in this case the obtained bound $\mathrm{rt}(G) = O(\log_d(n))$ asymptotically matches the diameter of $G$, which is clearly a lower bound on $\mathrm{rt}(G)$.

Throughout the proof we omit the use of floors and ceilings. Constants can be adjusted such that all inequalities hold with a sufficient margin to compensate for this.

\section{Nonblocking generalised matchings}

In this section we state some notions and results from the theory of \emph{wide-sense nonblocking networks}, the main machinery underlying our proof of Theorem \ref{thm:routing}.

\begin{definition} \label{def:nonblocking}
Given $d, t \in \mathbb{N}$, we say that a bipartite graph $G = (A \cup B, E)$ is \emph{$(d, t)$-nonblocking} if there exists a family $\cS$ of  subsets of $E$, called the \emph{safe} states, such that the following holds:
\begin{enumerate}[(P1)]
    \item $\emptyset \in \cS$,
    \item if $E'' \subseteq E'$ and $E' \in \cS$ then $E'' \in \cS$, and
    \item \label{prop:c} given $E' \in \cS$ of size $|E'| < t$ and a vertex $v \in A$ with $\deg_{E'}(v) < d$ (that is, $v$ is incident to less than $d$ edges in $E'$), there exists an edge $e = (v, w) \in E \setminus E'$ such that $E' \cup \{e\} \in \cS$ and $w$ is not incident to any edge in $E'$.
\end{enumerate}
\end{definition}

The following result is due to Feldman, Friedman, and Pippenger \cite[Proposition 1]{feldman88nonblocking}. It is proven in the same way as the more known result, at least within the combinatorics community, of Friedman and Pippenger \cite{friedman87expanding} on embeddings of trees in expanders.

\begin{lemma} \label{lemma:nonblocking}
Let $G = (A \cup B, E)$ be a bipartite graph and $a, d \in \mathbb{N}$. If for every $X \subseteq A$ of size $1 \le |X| \le 2a$ there are at least $2d|X|$ vertices in $B$ adjacent to some vertex in $X$, then $G$ is $(d, da)$-nonblocking.
\end{lemma}

The following lemma verifies the condition of Lemma \ref{lemma:nonblocking} for bipartite subgraphs of spectral expanders.

\begin{lemma} \label{lemma:spectral_nonblocking}
Let $G$ be an $(n, d, \lambda)$-graph with $\lambda < d / 72$, and let $V(G) = A \cup B$ be a partition of the vertex set. If every vertex $v \in A$ has at least $d/3$ neighbours in $B$, then the bipartite graph $G[A, B]$ is $(d / \lambda, n/24)$-nonblocking.
\end{lemma}
\begin{proof}
    We verify the conditions of Lemma \ref{lemma:nonblocking} with $a = \lambda n / (24 d)$. Consider some $X \subseteq A$ of size $|X| \le 2a \le \lambda n / (12 d)$. Suppose, towards a contradiction, that the set $Y \subseteq B$ of vertices adjacent to some vertex in $X$ is of size $|Y| < 2  (d / \lambda) |X|$. On the one hand, by the assumption of the lemma we have
    $$
        e(X, Y) \ge |X| d / 3.
    $$
    On the other hand, from \eqref{eq:mixing_lemma} together with the upper bound on $|Y|$, $|X|$ and $\lambda$, we have
    $$
        e(X, Y) \le d|X||Y| / n + \lambda \sqrt{|X||Y|} < 2 d^2 |X|^2 / (\lambda n) + |X| \sqrt{2 \lambda d} \le |X|d/3,
    $$
    thus we reach a contradiction.
\end{proof}

\section{Proof of Theorem \ref{thm:routing}} \label{sec:proof}

\begin{definition}
Given $k \in \mathbb{N}$, we say that a family of paths $\mathcal{P}$ in a graph $G$ is \emph{$k$-matching-switchable} if the following holds:
\begin{itemize}
    \item Each path $P \in \mathcal{P}$ is of the same odd length $\ell  = 2k+1$ (that is, each path has $\ell$ edges),
    \item All endpoints are distinct, and
    \item For each $z \in \{1, \ldots, k\}$, the set of edges $E_z(\mathcal{P})$, consisting of the $z$-th and $(\ell + 1 - z)$-th edge in each path in $\mathcal{P}$, forms a matching.
\end{itemize}
\end{definition}

The main property of a $k$-matching-switchable family $\mathcal{P}$, already used by Alon, Chung, and Graham \cite{alon94routing}, is that if we perform $2k+1$ rounds of routing where in round $z \le k$ we swap pebbles across $E_z(\mathcal{P})$, in round $k+1$ across edges corresponding to the  $(k+1)$-st edge from each path in $\mathcal{P}$ (note that this is indeed a matching), and in round $z > k + 1$ across $E_{2k + 2 - z}(\cP)$, then in the end all the pebbles other than the endpoints of paths in $\mathcal{P}$ remain where they were before, and the pebbles corresponding to endpoints of each path $P \in \mathcal{P}$ are swapped.

\begin{proof}[Proof of Theorem \ref{thm:routing}]
It is implicit in \cite[Theorem 2]{alon94routing} that every permutation can be represented as a composition of two permutations of order $2$. Thus it suffices to prove Theorem \ref{thm:routing} assuming $\pi$ is of order $2$, that is, $\pi^2$ is the identity permutation.

Using the probabilistic method, we first show that there exists a partition $V(G) = V_1 \cup V_2$ such that, for each vertex $v \in V(G)$, we have:
\begin{enumerate}[(i)]
    \item $v, \pi(v) \in V_i$ for some $i \in \{1,2\}$, and
    \item \label{prop:deg} $v$ has at least $d/3$ neighbours in both $V_1$ and $V_2$.
\end{enumerate}
For each cycle in $\pi$ (which is of length either $1$ or $2$ by the assumption on the order of $\pi$) toss an independent (fair) coin to decide whether to put the vertices of the cycle into $V_1$ or $V_2$. By the Lov\'asz Local Lemma (see \cite[Corollary 5.1.2]{probmethod}) and Chernoff-Hoeffding inequalities, the two properties are satisfied for each vertex simultaneously with positive probability, implying that a desired partition exists. Details are rather straightforward, thus we omit them.

Let $D = d / \lambda$. Suppose we are given a subset $W \subseteq V_i$ of size $|W| \le \eps n$, for $\eps = 1 / 72$, such that if $v \in W$ and $\pi(v) \neq v$ then $\pi(v) \not \in W$. We show that there exists a family of \emph{$k$-matching-routable} paths $\mathcal{P}_W = \{P_v\}_{v \in W}$, for $k = \log_D(n)$, where each $P_v \in \cP_W$ connects $v$ and $\pi(v)$. As observed earlier, these paths define a routing scheme which swaps the endpoints of paths in $\cP_W$ while leaving everything else intact. By greedily taking a new set $W \subseteq V_i$ ($|W| \le \eps n$) of vertices which are not yet routed and using the corresponding family $\mathcal{P}_W$ to swap the endpoints, every pebble reaches its destination after at most $\ell \cdot 2\lceil 1/\eps \rceil$ rounds, giving the desired bound.

Without loss of generality, suppose $W \subseteq V_1$. Consider $k = \log_D(n)$ bipartite graphs $G_z = G[A_z, B_z]$, for $z \in \{1, \ldots, k\}$, where $A_z = V_1$ if $z$ is odd and $A_z = V_2$ otherwise, and $B_z = V(G) \setminus A_z$. By Lemma \ref{lemma:spectral_nonblocking}, each $G_z$ is $(D, n/24)$-nonblocking. Let $(w_1, \ldots, w_t)$ be an arbitrary ordering of the vertices in $W$. We show, by induction on $i \in \{0, \ldots, t\}$, that there exists a family of $k$-matching-switchable paths $\mathcal{P}_i = (P_{w_j})_{j \le i}$, with each $P_{w_j}$ connecting $w_j$ to $\pi(w_j)$, such that $E_z(\cP_i)$ is a safe state in $G_z$ for each $z \in \{1, \ldots, k\}$. Note that the statement vacuously holds for $i = 0$. Suppose it holds for some $i < t$. We show that it then also holds for $i+1$ using the following procedure to find a new path $P_{w_{i+1}}$. Let $S_0 = \{w_{i+1}\}$ and $S'_0 = \{\pi(w_{i+1})\}$, and set $s_0 = 1$. For $z = 1, \ldots, k$, iteratively, let $M_z$ and $M_z'$ be the sets of $s_z := \min\{\eps n, s_{z-1} D\}$ edges in $G_z$ incident to $S_{z-1}$ and $S_{z-1}'$, respectively, such that each vertex in $B_z$ is incident to at most one edge in $E_z(\cP_i) \cup M_z \cup M_z'$  and $E_z(\cP_i) \cup M_z \cup M_z'$ is a safe state in $G_z$. This can be done by successively applying \ref{prop:c} (Definition \ref{def:nonblocking}) with $v \in \hat S_{z-1} = S_{z-1} \cup S_{z-1}'$, such that we ask for at most $D$ edges incident to each such vertex and in total for at most $s_z$ edges incident to each set $S_{z-1}$ and $S_{z-1}'$. Property \ref{prop:c} can indeed be applied as $|E_z(\cP_i)| < |W| \le \eps n$ and we further ask for at most $2 \eps n$ edges, thus we are always in a safe state with less than $3 \eps n \le n / 24$ edges (recall that $G_z$ is $(D, n/24)$-nonblocking). Finally, let $S_z$ ($S_z'$) be the endpoints of $M_z$ ($M_z'$) in $B_z$ (for the next iteration, recall that $B_z = A_{z+1}$). After all $k$ iterations are done, the choice of $k$ implies $|S_k| = |S_k'| = \eps n$, thus by \eqref{eq:mixing_lemma} there exists an edge $e$ between some $v_k \in S_k$ and $v_k' \in S_k'$. Now going backwards with $z = k, \ldots, 1$, let $m_z \in M_z$ be the unique edge incident to $v_z$, and $m_z' \in M_z'$ the unique edge incident to $v_z'$. Set $v_{z-1}$ ($v_{z-1}'$) to be the other endpoint of $m_z$ ($m_z'$), and proceed to the next iteration. The edges $m_1, m_2, \ldots, m_k, e, m_{k}', m_{k-1}', \ldots, m_1'$ then define a path $P_{w_{i+1}}$ such that $\cP_{i+1}$ satisfies the inductive hypothesis. This finishes the proof.
\end{proof}

\section{Concluding remarks}

The idea of using nonblocking properties of expanders to reach a large set of vertices, connect two of them, and then remove all unused ones is usually attributed to Daniel Johannsen \cite{johannsen} and is a  standard technique today. For other recent applications of this idea, see \cite{draganic22rolling,glebov,letzter2021size,montgomery19core,montgomery19tree}. The main difference between these applications and the presented one is that we do not aim to find paths which are entirely vertex-disjoint but rather only `locally', that is, they are disjoint with respect to one step of the routing. While this makes the task easier on the one hand, it also allows for more endpoints to be dealt with simultaneously, thus making the task harder on the other hand. Consequently, unlike most of the other results which rely directly on tree embeddings of Friedman and Pippenger \cite{friedman87expanding}, the `local' vertex-disjoint property is achieved through a repeated application of a related result of Feldman, Friedman, and Pippenger \cite{feldman88nonblocking} (Lemma \ref{lemma:nonblocking} in this note), with each step of the routing corresponding to an edge in a \emph{distinct} copy of a bipartite subgraph of $G$.

The proof of Lemma \ref{lemma:nonblocking} does not provide an efficient algorithm for finding an edge guaranteed by \ref{prop:c}. However, under somewhat stronger requirements on $G$, which are satisfied in our case, Aggarwal et al.\ \cite[Theorem 2.2.7]{aggarwal96optical} gave a polynomial time algorithm for finding such an edge, which in turn translates into a polynomial time algorithm for finding a routing scheme in Theorem \ref{thm:routing}.

\newpage
The following problem remains open: Is it true that if $G$ is a $d$-regular graph with the \emph{Cheeger constant} $h(G) \ge \eps$, for some $\eps > 0$, where
$$
    h(G) = \min\left\{ e(X, V(G) \setminus X) / |X| \colon X \subseteq V(G) \text{ and } |X| \le |V(G)|/2 \right\},
$$
then $\mathrm{rt}(G) = O_{d, \eps}(\log n)$? The best known bound is  $\mathrm{rt}(G) = O((d/\eps)^4 \log^2 n)$, due to Alon, Chung, and Graham \cite{alon94routing}.

{\small \bibliographystyle{abbrv} \bibliography{routing}}

\end{document}